\documentclass[11pt]{article}
\usepackage{geometry}
\usepackage{dsfont}

\usepackage{psfrag}

\usepackage{calrsfs}
\usepackage{mathrsfs}
\usepackage{pdfsync}
\usepackage{amsmath, amsthm, amssymb, amsfonts}
\usepackage[shortlabels]{enumitem}

\usepackage{url}
\usepackage{float}

\usepackage[usenames]{color}

\usepackage{tikz}

\usetikzlibrary{arrows,decorations.pathmorphing,backgrounds,positioning,fit,automata}

\usetikzlibrary{arrows}
\usetikzlibrary{petri}
\usetikzlibrary{topaths}

\usepackage{indentfirst,calc,euscript}
\usepackage{setspace}
\usepackage[reals]{layout}
\usepackage{xr}
\usepackage{amscd}

\usepackage{sgame}
\usepackage{subfigure}
\usepackage{slashbox}

\usepackage[colorlinks=true,breaklinks=true,bookmarks=true,urlcolor=blue,
     citecolor=blue,linkcolor=blue,bookmarksopen=false,draft=false]{hyperref}

\newcommand{\id}{\operatorname{Id}}

\newcommand{\ignore}[1]{}

\newtheorem{theorem}{Theorem}

\newtheorem{corollary}[theorem]{Corollary}

\newtheorem{lemma}[theorem]{Lemma}

\newtheorem{proposition}[theorem]{Proposition}

\theoremstyle{definition}

\newtheorem{definition}[theorem]{Definition}

\newtheorem{remark}[theorem]{Remark}

\newtheorem{assumption}{Assumption}

\newcounter{case}
\newenvironment{case}[1][Case \thecase]{\refstepcounter{case}  \begin{trivlist}
\item[\hskip \labelsep {\bfseries #1}]}{\end{trivlist}}

\numberwithin{equation}{section}

\numberwithin{theorem}{section}

\newcommand{\m}{\mathbb}

\author{Bruno Ziliotto\thanks{Paris Dauphine University, Place du Mar\'echal de Lattre de Tassigny, 75016 Paris,
France. \newline E-mail: ziliotto@math.cnrs.fr} }
\title{Tauberian theorems for general iterations of operators: applications to zero-sum stochastic games}
\begin{document}
\date{}
\maketitle
\bibliographystyle{plain}
\begin{abstract}
This paper proves several Tauberian theorems for general iterations of operators, and provides two applications to zero-sum stochastic games where the total payoff is a weighted sum of the stage payoffs. The first application is to provide conditions under which the existence of the asymptotic value implies the convergence of the values of the weighted game, as players get more and more patient. The second application concerns stochastic games with finite state space and action sets. This paper builds a simple class of asymptotically optimal  strategies in the weighted game, that at each stage play optimally in a discounted game with a discount factor corresponding to the relative weight of the current stage.
\end{abstract}
\section*{Introduction}
Zero-sum stochastic games were introduced by Shapley \cite{SH53}. In this model, two players repeatedly play a zero-sum game, that depends on an endogenous variable called \textit{state of nature}. At each stage, players choose a pair of actions. Actions and the current state determine both the stage payoff and the distribution according to which the new state is drawn. At the end of the stage, the new state of nature and the pair of actions played are publicly announced to both players.
\\
There are several ways to evaluate the total payoff in a stochastic game. For $n \in \m{N}^*$, the payoff in the $n-stage \ game$ is the Cesaro mean $\frac{1}{n} \sum_{m=1}^n g_m$, where $g_m$ is the payoff at stage $m \geq 1$. For $\lambda \in (0,1]$, the payoff in the $\lambda-discounted \ game$ is the Abel mean $\sum_{m \geq 1} \lambda(1-\lambda)^{m-1} g_m$. More generally, for a sequence of positive weights $(\theta_m)_{m \geq 1}$ summing to 1, the payoff in the $\theta$-weighted game is the weighted sum $\sum_{m \geq 1} \theta_m g_m$. Under mild conditions, the $n$-stage game, the $\lambda$-discounted game and the $\theta$-weighted game have a value, denoted respectively by $v_n$, $v_{\lambda}$ and $v_{\theta}$. 

A huge part of the literature focuses on the existence of the limit of $v_{n}$ when $n$ goes to infinity, of the limit of $v_{\lambda}$ when $\lambda$ goes to 0, and more generally of the limit of $v_{\theta}$ when $\sup_{m \geq 1} \theta_m$ goes to 0. When the state space and the action sets are finite, Bewley and Kohlberg have proved that $(v_n)$ and $(v_{\lambda})$ converge to the same limit, and a direct consequence of the existence of the uniform value by Mertens and Neyman \cite{MN81} is that more generally, $(v_{\theta})$ converges when $(\theta_m)_{m \geq 1}$ is decreasing and $\theta_1$ goes to 0. Without this finiteness assumption, several counterexamples have been pointed out. Vigeral \cite{vigeral13} has provided an example of a stochastic game with finite state space and compact action sets in which neither $(v_n)$ nor $(v_{\lambda})$ converges. Ziliotto \cite{Z13} has provided an example of a stochastic game with compact state space, finite action sets and many other regularity properties, in which $(v_n)$ and $(v_{\lambda})$ fail to converge. Many papers prove convergence of $(v_n)$, $(v_{\lambda})$ and $(v_{\theta})$ to the same limit in specific models (see for instance the recent surveys \cite{LS15}, \cite{JN16} and \cite{SZ16}). A natural question is to ask whether there is a link between the convergence of $(v_n)$, $(v_{\lambda})$ and $(v_{\theta})$\footnote{In a continuous-time framework, Oliu-Barton and Vigeral \cite{BV13} and Khlopin \cite{K14,K15} address this question in the optimal control and differential game setting.}. Ziliotto \cite{Z15} has proved that in a very general stochastic game model, with possibly infinitely many states and actions, $(v_n)$ converges uniformly if and only if $(v_{\lambda})$ converges uniformly\footnote{Values depend on the initial state, thus $(v_n)$ and $(v_{\lambda})$ map the state space to the reals. Note that if uniform convergence is replaced by pointwise convergence, such a result does not hold, even in the 1-Player case (see Lehrer and Sorin \cite{LS92}).}.
\\  
This paper aims at generalizing such a result to a more general family of values $(v_{\theta})$. Note that for recursive games (stochastic games in which the stage payoff is 0 in nonabsorbing states), Li and Venel \cite{LV16} have proved that if $(v_n)$ or $(v_\lambda)$ converges uniformly, then $(v_{\theta})$ converges uniformly when $\theta$ is decreasing and $\theta_1$ goes to 0. In a dynamic programming framework (one player), Monderer and Sorin \cite{MS93} have proved the convergence property, but for a more restricted class of decreasing evaluations. 
\\
The contribution of this paper is twofold. In the same stochastic game model as in \cite{Z15}, it provides conditions under which the uniform convergence of $(v_n)$ or $(v_{\lambda})$ implies the uniform convergence of $(v_{\theta})$.  
Second, in the case where the state space and the action sets are finite, it proves that the following \textit{discounted strategy} is asymptotically optimal in the $\theta$-weighted game, as $\theta_1$ goes to 0: at each stage $m \geq 1$, play optimally in the discounted game with discount factor $\theta_m/(\sum_{m' \geq m} \theta_{m'})$\footnote{Let us mention that in a recent independent work, Oliu-Barton has obtained similar results, using a different approach. These results will be published in another paper.}. Such a discount factor corresponds to the weight of stage $m$ relative to the weight of future stages. This result is new even when $\theta$ is a $n$-stage evaluation.
Finally, this paper provides an example that illustrates the sharpness of the first result.  

The proof of the results rely on the operator approach, introduced by Rosenberg and Sorin \cite{RS01}. This approach builds on the fact that the value of the $\theta$-weighted game satisfies a functional equation, called the Shapley equation (see \cite{SH53}). The properties of the associated operator can be exploited to infer convergence properties of $(v_{\theta})$ (see \cite{RS01}). This paper first proves several Tauberian theorems in an operator setting, and apply them to stochastic games to get the first result. Surprisingly, the proof of the second result also follows from a Tauberian theorem for operators. It is due to the fact that the payoff $w_{\theta}$ guaranteed by a discounted strategy in the $\theta$-weighted game satisfies a functional equation similar to a Shapley equation.

This paper is organized as follows. Sections 1 and 2 state and prove several Tauberian theorems for operators. Section 3 applies these theorems to stochastic games. Section 4 provides an example that shows the sharpness of the results.

\section{Operator approach: main results}
If $\left(C,\mathcal{C}\right)$ is a Borel subset of a Polish space, we denote by $\Delta(C)$ the set of probability measures on $C$, equipped with the weak$^*$ topology. In particular, the set $\Delta(\m{N}^*)$ identifies with the set of positive real sequences $(\theta_m)_{m \geq 1}$ such that $\sum_{m \geq 1} \theta_m=1$. Let $\mathcal{D}:=\left\{\theta \in \Delta(\m{N}^*) \ | \ \forall \ m \geq 1, \ \theta_{m+1} \leq \theta_m \right\}$. The sequence $\theta \in \m{R}^{\m{N}^*}$ such that $\theta_m=0$ for all $m \geq 1$ is denoted by  $\mathbf{0}$. 
\begin{definition} \label{shift}
Let $\theta \in \Delta(\m{N}^*) \cup \left\{\mathbf{0}\right\}$. 
\begin{itemize}
\item
The \textit{shift of $\theta$} is the sequence $\hat{\theta}$ defined by $\hat{\theta}:=\mathbf{0}$ if $\theta_1= 1$, and otherwise
\begin{equation*}
\forall \ m \in \m{N}^*, \quad \hat{\theta}_m:=\theta_{m+1}(1-\theta_1)^{-1}.
\end{equation*}
\item
Let $r \in \m{N}^*$. The \textit{$r$-shift} of $\theta$ is the sequence $\theta^r \in \Delta(\m{N}^*)$ defined recursively by $\theta^1=\theta$ and 
\begin{equation*}
\theta^r:=\widehat{\theta^{r-1}}.
\end{equation*}
Note that if $\theta^r \neq \mathbf{0}$, then for all $m \geq 1$, 
\begin{equation*}
\theta^r_m=\frac{\theta_{m+r-1}}{\sum_{m'\geq r} \theta_{m'}}.
\end{equation*}
\end{itemize}
\end{definition}
Let $(X, \left\| . \right\|$) be a Banach space, and a map $\Psi:[0,1] \times X \rightarrow X$ that satisfies 
\begin{equation} \label{nonexp}
\forall \lambda \in [0,1], \ \forall \ (f,g) \in X^2, \quad \left\|\Psi(\lambda,f)-\Psi(\lambda,g) \right\| \leq (1-\lambda)\left\|f-g\right\|,
\end{equation}
and
\begin{equation} \label{reg}
\exists C>0, \ \forall \lambda \in [0,1], \ \forall \ f \in X, \quad \left\|\Psi(\lambda,f) \right\| 
\leq C\lambda+(1-\lambda) \left\|f \right\|. 
\end{equation}
\begin{proposition} \label{existence}
There exists a unique family $(v_{\theta})_{\theta \in \Delta(\m{N}^*)\cup \left\{\mathbf{0}\right\}}$ such that: 
\\
$(v_{\theta})_{\theta \in \Delta(\m{N}^*)\cup \left\{\mathbf{0}\right\}}$ is bounded, and
\begin{equation} \label{recop}
\forall \ \theta \in \Delta(\m{N}^*) \cup \left\{\mathbf{0}\right\}, \quad v_{\theta}=\Psi(\theta_1,v_{\widehat{\theta}}).
\end{equation}
\end{proposition}
  When $\theta$ is taken among decreasing sequences, this proposition is a direct consequence of Neyman and Sorin \cite[Subsection 3.2]{NS10}).
\begin{remark} \label{rkshapley}
A class of operators that play a key role in this paper is the following. Let $K$ be any set, and $X$ be the set of bounded real-valued functions defined on $K$. Consider two sets $S$ and $T$, and a family of linear forms $(P_{k,s,t})_{(k,s,t) \in K \times S\times T}$ on $X$, such that for all $(k,s,t)$, $P_{k,s,t}$ is of norm smaller than 1. Let $g:K \times S \times T \rightarrow \m{R}$ be a bounded function. Define $\Psi :[0,1] \times X \rightarrow X$ by $\Psi(\lambda,f)(k):=\sup_{s \in S} \inf_{t \in T} \left\{\lambda g(k,s,t)+(1-\lambda) P_{k,s,t}(f) \right\}$, for all $f \in X$ and $k \in K$. Note that $\Psi$ satisfies (\ref{nonexp}) and (\ref{reg}). This class includes \textit{Shapley operators} (see Neyman and Sorin \cite[p.397-415]{NS03b}): this corresponds to the case where $K$ is the state space of some zero-sum stochastic game, $S$ (resp. $T$) is the set of strategies of Player 1 (resp. 2), $k$ is the current state, and $P_{k,s,t}(f)$ is the expectation of $f(k')$ under strategies $s$ and $t$, where $k'$ is the state at next stage. Under standard assumptions, for all $\theta \in \Delta(\m{N}^*)$, $v_{\theta}$ is the value of the stochastic game in which the weight of the payoff at stage $m$ is $\theta_m$, for all $m \geq 1$. This fact will be useful in Section \ref{app}.
\end{remark}
\begin{remark} \label{rknonexp}
Consider a 1-Lipschitz mapping $\Phi: X \rightarrow X$ and 
\\
define $\Psi(\lambda,f):=\lambda \Phi(\lambda^{-1}(1-\lambda) f)$, for all $\lambda \in (0,1]$ and $f \in X$. Extend $\Psi$ to $[0,1] \times X$ by continuity. The operator $\Psi$ satisfies (\ref{nonexp}) and (\ref{reg}) for $C=\left\|\Phi(0)\right\|$. Thus, the framework considered in this paper is slightly more general than the nonexpansive operator setting that is standard in the literature. This more general setting will be useful in Subsection \ref{secfrac}. 
\end{remark}

We emphasize two particular families of evaluations $\theta \in \Delta(\m{N}^*)$:
\begin{itemize}
\item
For $n \in \m{N}^*$, the element $v_n \in X$ is defined by $v_n:=v_{\theta(n)}$, where $\theta(n) \in \mathcal{D}$ is defined in the following way: for all $m \in \m{N}^*$, $\theta(n)_m=1_{m \leq n} n^{-1}$.
\item
For $\lambda \in (0,1]$, the element $v_{\lambda} \in X$ is defined by $v_{\lambda}:=v_{\theta(\lambda)}$, where $\theta(\lambda) \in \mathcal{D}$ is defined in the following way: for all $m \in \m{N}^*$, 
$\theta(\lambda)_m=\lambda(1-\lambda)^{m-1}$.
\end{itemize}
By definition, we have
\begin{equation*}
v_{\lambda}=\Psi(\lambda,v_{\lambda}),
\end{equation*}
and
\begin{equation*}
v_n=n^{-1} \Psi\left(\frac{1}{n},v_{n-1}\right).
\end{equation*}
Some of the results will be stated under the following assumption, which is stronger than (\ref{reg}):
\begin{assumption} \label{assop}
There exists $C>0$ such that for all $\alpha,\beta \in [0,1]$, for all $\lambda, \lambda' \in  [0,1]$, 
for all $f,g \in X$,
\begin{equation*}
\left\| \alpha \Psi(\lambda,f)-\beta \Psi(\lambda',g) \right\| \leq C \left| \alpha \lambda-\beta\lambda' \right|
+\left\|\alpha(1-\lambda) f-\beta(1-\lambda') g\right\|.
\end{equation*}
\end{assumption}
\begin{remark}
Shapley operators (see Remark \ref{rkshapley}) satisfy Assumption \ref{assop}.
\end{remark}
The following proposition stems from the proof of Ziliotto \cite[Theorem 1.2]{Z15} (for further details, see Section \ref{secproof}).
\begin{proposition} \label{taub}
Under Assumption 1, the following statements are equivalent:
\begin{enumerate}[(a)]
\item The sequence $(v_n)_{n \geq 1}$ converges when $n$ goes to infinity.
\item The mapping $\lambda \rightarrow v_{\lambda}$ has a limit when $\lambda$ goes to 0.
\end{enumerate}
Moreover, when these statements hold, we have $\lim_{n \rightarrow+\infty} v_n=\lim_{\lambda \rightarrow 0} v_{\lambda}$. 
\end{proposition}
When $(a)$, or equivalently $(b)$ is satisfied, the common limit is called \textit{asymptotic value}, and the operator $\Psi$ is said to have an asymptotic value. One of the main goals of this paper is to determine under which conditions on the sequences of weights does the existence of the asymptotic value implies the convergence of $(v_{\theta})$, as $\sup_{m \geq 1} \theta_m$ goes to 0. This leads to consider the following family of sequences of weights:

\begin{definition}
Let $p \in \m{N}^*$. A sequence $\theta \in \Delta(\m{N}^*)$ is \textit{$p$-piecewise constant} if there exists
$a_1,...,a_p \in [0,1]$ and $m_1,...,m_{p+1} \in \m{N}^*$ such that $m_1=1$ and for all $p' \in \left\{1,...,p \right\}$,
for all $m \in \left\{m_{p'},...,m_{p'+1}-1\right\}$, $\theta_m=a_{p'}$, and $\theta_m=0$ for all $m \geq m_{p+1}$.
\end{definition}
For $p \in \m{N}^*$, denote ${\Theta}^p$ the set of sequences that are $p$-piecewise constant. Note that $\Theta^1=\left\{\theta(n), n \in \m{N}^* \right\}$, and that for $p \geq 2$, $\Theta^p$ is not included in $\mathcal{D}$. 
\\
For $\theta, \theta' \in \Delta(\m{N}^*)$, define
\begin{equation*}
\left\|\theta-\theta'\right\|_1:=\sum_{m \geq 1} |\theta_{m}-\theta'_m|.
\end{equation*}
For $\Theta \subset \Delta(\m{N}^*)$, denote $D(\theta,\Theta):=\inf_{\theta' \in\Theta} \left\|\theta-\theta' \right\|_1$ and
\\ 
$I_p(\theta):=\max \left\{\sup_{m \geq 1} \theta_m, D(\theta,\Theta^p) \right\}$. 
\\
In the stochastic game framework, the quantity $I_p(\theta)$ is a measure of the ``impatience" of players. When $I_p(\theta)$ is small, this means first that players are very patient, second that the weight that they put on each stage does not vary too much: it can make at most $p$ significant jumps.


\begin{theorem} \label{taubgen}
Suppose that Assumption \ref{assop} holds, and that $\Psi$ has an asymptotic value $v^*$.

Then for all $\epsilon>0$, for all $p \in \m{N}^*$, there exists $\alpha>0$ such that for all $\theta \in \Delta(\m{N}^*)$,
\begin{equation*}
 I_p(\theta) \leq \alpha \Rightarrow
\left\|v_{\theta}-v^* \right\|  \leq \epsilon.
\end{equation*}
\end{theorem}
As an example, let us give a family of evaluations to which the theorem applies.
\\
Let $p \in \m{N}^*$. A sequence $\theta \in \Delta(\m{N}^*)$ is \textit{$p$-piecewise-discounted} if there exists
$a_1,...,a_p \in \m{R}_+$, $\lambda_1,...,\lambda_p \in [0,1]$ and $m_1,...,m_p \in \m{N}^*$ such that $m_1=1$ and for all $p' \in \left\{1,...,p-1\right\}$,
for all $m \in \left\{m_{p'},...,m_{p'+1}-1\right\}$, $\theta_m=
a_{p'} (1-\lambda_{p'})^{m-m_{p'}}$, and for all $m \geq m_p$, $\theta_m=a_{p} (1-\lambda_{p})^{m-m_{p}}$. Denote $\Theta^p_d$ the set of sequences that are $p$-piecewise-discounted. Note that $\Theta^1=\left\{\theta(\lambda), \lambda \in (0,1] \right\}$, and that $\Theta^p \subset \Theta^{p+1}_d$. 
\begin{corollary} \label{cor}
Suppose that Assumption \ref{assop} holds, and that $\Psi$ has an asymptotic value $v^*$. Then for all $\epsilon>0$, for all $p \in \m{N}^*$, there exists $\alpha>0$ such that for all $\theta \in \Theta^p_d$,
\begin{equation*}
 \sup_{m \geq 1} \theta_m \leq \alpha \Rightarrow
\left\|v_{\theta}-v^* \right\|  \leq \epsilon.
\end{equation*} 
\end{corollary} 
As we shall see in Section \ref{secex}, this result does not hold when one replaces $\Theta^p_d$ by $\mathcal{D}$. For such a result to hold, additional assumptions on the family $(v_{\lambda})$ are required, as shown by the next theorem.
\begin{definition}
The family $(v_{\lambda})$ has \textit{bounded variation} if for all $\epsilon>0$, 
there exists $\beta \in (0,1]$ such that for all decreasing sequences 
$(d_r)_{r \geq 1} \in (0,\beta]^{\m{N}^*}$,  
\begin{equation*}
\sum_{r=1}^{+\infty} \left\|v_{d_{r+1}}-v_{d_r} \right\| \leq \epsilon.
\end{equation*}
\end{definition}
\begin{theorem} \label{taubbv}
Assume that one of the following assumptions holds:
\begin{enumerate}[(a)]
\item \label{bv1}
$(v_{\lambda})$ has bounded variation and Assumption 1 holds.
\item \label{bv2}
There exists $s>0$ and $\beta \in (0,1]$ such that for all $\lambda,\lambda' \in (0,\beta]$, 
\begin{equation*}
\left\|v_{\lambda}-v_{\lambda'}\right\| \leq C\left|\lambda^s-{\lambda'}^s \right|.
\end{equation*}
\end{enumerate}
Then $\Psi$ has an asymptotic value $v^*$. Moreover, for all $\epsilon>0$, there exists $\alpha>0$ such that for all $\theta \in \mathcal{D}$,
\begin{equation*}
 \sup_{m \geq 1} \theta_m \leq \alpha \Rightarrow
\left\|v_{\theta}-v^* \right\|  \leq \epsilon.
\end{equation*}
\end{theorem}
\begin{remark}
The fact that Assumption \ref{bv1} or \ref{bv2} imply the existence of the asymptotic value was already known by \cite[Theorem C.8, p.177]{sorin02b}, or by Proposition \ref{taub}. Note also that when $\Psi$ is the Shapley operator of some stochastic game with finite state space and finite action sets, then \ref{bv1} and \ref{bv2} are satisfied (see \cite{BK76}).
\end{remark}


\section{Proofs} \label{secproof}
This section is dedicated to the proofs of the results. 
\subsection{Iterated operators}
Iterating Equation (\ref{recop}) in a proper way is one of the key ingredients to the proofs, thus we define the following family of iterated operators: 
\begin{definition}
Let $n \in \m{N}$ and $\theta \in \Delta(\m{N}^*) \cup \left\{\mathbf{0}\right\}$. The operator $\Psi^n_{\theta}: X \rightarrow X$ is defined recursively by $\Psi^{0}_{\theta}:=\id$, and for $n \geq 1$,
\begin{equation*}
\forall \ f \in X, \quad \Psi^n_{\theta}(f):=\Psi(\theta_1,\Psi^{n-1}_{\hat{\theta}}(f)).
\end{equation*}
\end{definition}
When $\theta=\theta(\lambda)$ is a discounted evaluation, we write $\Psi^n_{\lambda}$ instead of $\Psi^n_{\theta(\lambda)}$. 
\\
The following lemma establishes important properties for the family of operators $(\Psi^n_{\theta})$. 
\begin{lemma} \label{iterate} 
Let $f,g \in X$, $\lambda, \lambda' \in [0,1]$, $n \in \m{N}^*$, and $\theta, \theta' \in \Delta(\m{N}^*) \cup \left\{\mathbf{0} \right\}$. The following assertions hold:
\begin{enumerate}[(i)]
\item \label{it1}
\begin{equation*}
\left\|\Psi^n_{\theta}(f)-\Psi^n_{\theta}(g)\right\| \leq \left(\sum_{m \geq n+1} \theta_m \right) \left\|f-g \right\|
\end{equation*}
\item \label{bounded}
\begin{equation*}
\left\|\Psi^n_{\theta}(f) \right\| \leq C \sum_{m=1}^n \theta_m+\sum_{m \geq n+1} \theta_m \left\|f\right\|
\end{equation*}
\item \label{it2}
Under Assumption \ref{assop},
\begin{equation*}
\left\|\Psi^n_{\theta}(f)-\Psi^n_{\theta'}(f)\right\| \leq C \sum_{m=1}^n \left|\theta_m-\theta'_m\right|+\left|\sum_{m \geq n+1} (\theta_m-\theta'_m) \right| \left\| f \right\|.
\end{equation*}
\end{enumerate}
\end{lemma}
\begin{proof}
We introduce the following piece of notation: for $\theta \in \Delta(\m{N}^*) \cup \left\{\mathbf{0}\right\}$ and $r \in \m{N}^*$, let 
\begin{equation*}
\Pi_r(\theta):=\prod_{r'=1}^{r-1} (1-\theta^{r'}_1)=\sum_{m \geq r} \theta_m.
\end{equation*} 
\begin{enumerate}[(i)]
\item 
By Equation (\ref{nonexp}), we have
\begin{eqnarray*}
\left\|\Psi^n_{\theta}(f)-\Psi^n_{\theta'}(g)\right\|&=&
\left\|\Psi(\theta_1,\Psi^{n-1}_{\hat{\theta}}(f))-\Psi(\theta_1,\Psi^{n-1}_{\hat{\theta}}(g))
\right\|
\\
&\leq& (1-\theta_1) \left\|\Psi^{n-1}_{\hat{\theta}}(f)-\Psi^{n-1}_{\hat{\theta}}(g)\right\|.
\end{eqnarray*}
Iterating this inequality yields
\begin{eqnarray*}
\left\|\Psi^n_{\theta}(f)-\Psi^n_{\theta}(g)\right\|&\leq& \prod_{m=1}^n(1-\theta^{m}_1) \left\|f-g\right\|
\\
&=& \left(\sum_{m \geq n+1} \theta_m \right) \left\|f-g\right\|.
\end{eqnarray*}
\item 
Applying Equation (\ref{reg}), we get
\begin{eqnarray*}
\left\|\Psi^n_{\theta}(f) \right\| \leq C \theta_1+(1-\theta_1) \left\|\Psi^{n-1}_{\hat{\theta}}(f) \right\|,
\end{eqnarray*}
and the result follows by induction.
\item
For $s \in \left\{0,1,...,n\right\}$, define 
\begin{equation*}
\Phi(s):=\left\| \Pi_{n-s+1}(\theta) \Psi_{\theta^{n-s+1}}^{s}(f)-{\Pi}_{n-s+1}(\theta') \Psi_{{\theta'}^{n-s+1}}^{s}(f) \right\|.
\end{equation*}
Let us prove by induction that for all $s \in \left\{0,1,...,n \right\}$,
\begin{equation} \label{ineqii}
\Phi(s) \leq C \sum_{m=n-s+1}^n \left|\theta_m-\theta'_m\right|+\left|\sum_{m \geq n+1} (\theta_m-\theta'_m) \right| \left\| f \right\|.
\end{equation}
We have $\Phi(0)=\left| \Pi_{n+1}(\theta)-\Pi'_{n+1}(\theta') \right| \left\|f \right\|$, thus the above inequality holds for $s=0$. 
\\
Assume $s \geq 1$. Let $\alpha:=\Pi_{n-s+1}(\theta)$, $\beta:=\Pi'_{n-s+1}(\theta')$,  $\lambda=\theta^{n-s+1}_1$, 
$\lambda'=\theta'^{n-s+1}_1$, $f:= \Psi_{\theta^{n-s+2}}^{s-1}(f)$ and $g:=\Psi_{{\theta'}^{n-s+2}}^{s-1}(f)$. Applying Assumption \ref{assop}, we get
\begin{eqnarray*}
\Phi(s)&=&\left\|\alpha \Psi(\lambda,f)-\beta \Psi(\lambda',g)\right\|
\\
&\leq& C \left|\alpha \lambda-\beta \lambda'\right|+\left\| \alpha(1-\lambda) f- \beta (1-\lambda') g \right\|
\\
&=& C \left|\theta_{n-s+1}-\theta'_{n-s+1} \right|+\Phi(s-1).
\end{eqnarray*}
By induction hypothesis, we deduce that inequality (\ref{ineqii}) holds for all $s \in \left\{0,1,...,n \right\}$. Taking $s=n$ yields \ref{it2}.
\end{enumerate}
\end{proof}
\subsection{Proofs of Propositions \ref{existence} and \ref{taub}}
\subsubsection{Proposition \ref{existence}}
Let us start with the proof of Proposition \ref{existence}, that is, the existence and uniqueness
of the family $(v_{\theta})$ defined by Equation $(\ref{recop})$.
Let $\theta \in \Delta(\m{N}^*) \cup \left\{\mathbf{0}\right\}$. 
For $n \geq 1$, define 
\begin{equation*}
f^n_{\theta}:=\Psi^n_{\theta}(0).
\end{equation*}
Note that by Lemma \ref{iterate} \ref{bounded}, for all $n \in \m{N}^*$, $\left\|f^n_{\theta}\right\| \leq C$. 
Let $n \geq 1$ and $l \geq 0$. Using Lemma \ref{iterate} \ref{it1} and \ref{bounded}, we get
\begin{eqnarray*}
\left\|f^{n+l}_{\theta}-f^n_{\theta} \right\|&=&
\left\|\Psi^n_{\theta}(f^{l}_{\theta^{n+1}})-\Psi^n_{\theta}(0) \right\|
\\
&\leq& \left(\sum_{m \geq n+1} \theta_m \right) \left\|f^{l}_{\theta^{n+1}} \right\|
\\
&\leq& C \sum_{m \geq n+1} \theta_m.
\end{eqnarray*}
It follows that $(v_{\theta^n})$ is a Cauchy sequence in the Banach space $X$. Consequently, it converges. Define $v_{\theta}:=\lim_{n \rightarrow+\infty} v_{\theta^n}$. 

By definition, we have
\begin{equation*}
f^n_{\theta}=\Psi(\theta_1,f^{n-1}_{\hat{\theta}}).
\end{equation*}
Because $\Psi(\theta_1,.)$ is continuous, passing to the limit gives that $v_{\theta}$ satisfies Equation (\ref{recop}). Moreover, $(v_{\theta})_{\theta \in \Delta(\m{N}^*) \cup \left\{\mathbf{0}\right\}}$ is bounded by $C$.

Now let us prove the uniqueness of the family $(v_{\theta})_{\theta \in \Delta(\m{N}^*)\cup \left\{\mathbf{0}\right\}}$. Let $(w_{\theta})_{\theta \in \Delta(\m{N}^*)\cup \left\{\mathbf{0}\right\}}$ satisfying the conditions of Proposition \ref{existence}. Using Lemma \ref{iterate} \ref{it1}, we deduce that for all $\theta \in \Delta(\m{N}^*)$ and $n \geq 1$,
\begin{eqnarray*}
\left\|v_{\theta}-w_{\theta} \right\| &\leq& \left(\sum_{m \geq n+1} \theta_m \right) \left\|v_{\theta^n}-w_{\theta^n}\right\|
\\
&\leq&  \left(\sum_{m \geq n+1} \theta_m\right) \left(\left\|v_{\theta^n}\right\|+\left\|w_{\theta^n}\right\|\right).
\end{eqnarray*}
Because $(v_{\theta})_{\theta \in \Delta(\m{N}^*)\cup \left\{\mathbf{0}\right\}}$ and $(w_{\theta})_{\theta \in \Delta(\m{N}^*)\cup \left\{\mathbf{0}\right\}}$ are bounded, taking $n$ to infinity yields that $v_{\theta}=w_{\theta}$ for all $\theta \in \Delta(\m{N}^*) \cup \left\{\mathbf{0} \right\}$. 
\subsubsection{Proposition \ref{taub}}
As mentioned earlier, this proposition is a direct consequence of the proof in \cite[Theorem 1.2]{Z15}. 
Indeed, the only difference between Theorem 1.2 in \cite{Z15} and Proposition \ref{taub} is that the theorem deals with a more specific class of operators: as in Remark \ref{rknonexp}, the operator $\Psi$ is of the form
\begin{equation*}
\forall \lambda \in (0,1] \quad \forall f \in X \quad \Psi(\lambda,f)=\lambda \Phi(\lambda^{-1}(1-\lambda) f),
\end{equation*}
where $\Phi:X \rightarrow X$ is a 1-Lipschitz operator. The key point is that the proof of the theorem only relies on the properties of the iterated operator stated in \cite[Lemma 1]{Z15}. Thanks to Lemma \ref{iterate}, these properties are also satisfied in our more general framework. Consequently, the same proof applies as well.

\subsection{Proof of Theorem \ref{taubgen} and Corollary \ref{cor}}
In this subsection, Assumption \ref{assop} is in force. Let us start with a lemma, that is important to prove both Theorem \ref{taubgen} 
and Corollary \ref{cor}. 
\begin{lemma} \label{comp}
Let $\theta, \theta' \in \Delta(\m{N}^*) \cup \left\{\mathbf{0}\right\}$. The following assertion holds:
\begin{equation*} 
\left\|v_{\theta}-v_{\theta'}\right\| \leq C \left\|\theta-\theta'\right\|_1.
\end{equation*}
In particular,
\begin{equation*} 
\left\|v_{\theta}\right\| \leq C.
\end{equation*}
\end{lemma}
\begin{proof}
Taking $f=0$ in Lemma \ref{iterate} \ref{it2}, we get that for all $n \in \m{N}$,
\begin{equation*}
\left\|\Psi^n_{\theta}(0)-\Psi^n_{\theta'}(0)\right\|
\leq C \left\|\theta-\theta'\right\|_1.
\end{equation*}
Taking $n$ to infinity in the above equation yields the first inequality. Taking $\theta'=\mathbf{0}$ in the first inequality implies that $\left\|v_{\theta}\right\| \leq C$ (this is also a consequence of the proof of Proposition \ref{existence}).
\end{proof}
\subsubsection{Theorem \ref{taubgen}}
Now let us prove Theorem \ref{taubgen}.
By the previous lemma, it is enough to prove that for any $p \in \m{N}^*$, the following property $\mathcal{A}(p)$ holds: 
\\
for all $\epsilon>0$, there exists
$\alpha>0$ such that for all $\theta \in \Theta^p$, 
\begin{equation} \label{rank}
\sup_{m \geq 1} \theta_m \leq \alpha \Rightarrow
\left\|v_{\theta}-v^* \right\|  \leq \epsilon.
\end{equation}
Let us prove the result by induction on $p \geq 1$. By assumption, $(v_n)$ converges to $v^*$, thus $\mathcal{A}(1)$ holds. Assume now that $\mathcal{A}(p)$ holds for some $p \geq 1$. 
\\
Let $\theta \in \Theta^{p+1}$: there exists
$a_1,...,a_p \in [0,1]$ and $m_1,...,m_{p+1} \in \m{N}^*$ such that $m_1=1$ and for all $p' \in \left\{1,...,p \right\}$, for all $m \in \left\{m_{p'},...,m_{p'+1}-1\right\}$, $\theta_m=a_{p'}$, and $\theta_m=0$ for all $m \geq m_{p+1}$. 
Thanks to Lemma \ref{comp}, we may assume without loss of generality that $0<a_1<1$ (indeed we can perturbate the weights so that they become positive, and this will not significantly change the value $v_{\theta}$). Let $\theta':=\theta^{m_2}$. 
Let $n_1$ be the integer part of $1/a_1$.
We have
\begin{equation} \label{fix11}
v_{\theta}=\Psi^{m_2-1}_{\theta}(v_{\theta'}),
\end{equation}
and
\begin{equation} \label{fix12}
v_{n_1}=(n_1)^{-1}\Psi^{m_2-1}((n_1-m_2+1)v_{n_1-m_2+1}).
\end{equation}
Moreover, 
\begin{eqnarray} \label{tri}
\left\|v_{\theta}-v_{n_1} \right\| &\leq& \left\|v_{\theta}-\Psi^{m_2-1}_{\theta}(v_{n_{1}-m_2+1}) \right\|+
\left\|\Psi^{m_2-1}_{\theta}(v_{n_{1}-m_2+1})-v_{n_1} \right\|.
\end{eqnarray}
Equations (\ref{fix11}), (\ref{fix12}) and Lemma \ref{iterate} \ref{it1} imply that
\begin{eqnarray} \label{tri1}
\left\|v_{\theta}-\Psi^{m_2-1}_{\theta}(v_{n_{1}-m_2+1}) \right\| &\leq& (1-(m_2-1)a_1) \left\|v_{\theta'}-v_{n_{1}-m_2+1} \right\|  
\end{eqnarray}
and
\begin{eqnarray}
\nonumber
\left\|\Psi^{m_2-1}_{\theta}(v_{n_{1}-m_2+1})-v_{n_1} \right\| &\leq&
C(m_2-1)(1/n_1-a_1)+(m_2-1)(1/n_1-a_1) \left\|v_{n_1-m_2+1}\right\|
\\
\nonumber
&\leq& 2C (m_2-1)(1/n_1-a_1)
\\
\nonumber
&\leq& 2C (m_2-1)a_1^2/(1-a_1)
\\ \label{tri2}
&\leq& 2C a_1/(1-a_1). 
\end{eqnarray}
Note that the above quantity vanishes as $\sup_{m \geq 1} \theta_m$ goes to 0. 
\\
Let $\epsilon>0$, and let $\alpha$ be as in (\ref{rank}). Let us prove that there exists $\alpha'>0$ such that 
\begin{equation*}
\sup_{m \geq 1} \theta_m \leq \alpha'
\
\Rightarrow
\
(1-(m_2-1)a_1) \left\|v_{\theta'}-v_{n_{1}-m_2+1} \right\| \leq \epsilon/2.
\end{equation*}
Thanks to Equations (\ref{tri}), (\ref{tri1}) and (\ref{tri2}), this is enough to prove that $\mathcal{A}(p+1)$ holds. 
We discriminate between the two following cases:
\begin{case}{$(m_2-1)a_1 \geq 1-(\epsilon/4C)$}
\end{case}
In this case, 
\begin{eqnarray*}
(1-(m_2-1)a_1) \left\|v_{\theta'}-v_{n_1-m_2+1} \right\| &\leq& (\epsilon/4C) \left(\left\|v_{\theta'}\right\|
+\left\|v_{n_1-m_2+1} \right\| \right) \leq \epsilon/2.
\end{eqnarray*}
\begin{case}{$(m_2-1)a_1 \leq 1-(\epsilon/4C)$}
\end{case}
Because $(v_n)$ converges to $v^*$, there exists $n_0 \in \m{N}^*$ such that for all $n \geq n_0$,
\begin{equation*}
\left\|v_n-v^*\right\| \leq \epsilon/4. 
\end{equation*}
Note that $n_1-m_2+1 \geq 1/a_1-1-(1-\epsilon/4C)/a_1=\epsilon/(4C a_1)-1$. 
Moreover, 
\\
$\theta'_1=a_2/(1-(m_2-1)a_1) \leq (4C/\epsilon) \sup_{m \geq 1} \theta_m$. 
Take $\alpha'=\min(\epsilon/(4C(n_0+1)),\alpha \epsilon/4C)$. Hence
\begin{equation*} 
\sup_{m \geq 1} \theta_m \leq \alpha' \Rightarrow
(1-(m_2-1)a_1) \left\|v_{\theta'}-v_{n_{1}-m_2+1} \right\| \leq \epsilon/2,
\end{equation*}
which concludes the proof.
\subsubsection{Corollary \ref{cor}}
The main idea is to approximate piecewise-discounted evaluations by piecewise-constant evaluations. Indeed, thanks to Lemma \ref{comp}, in order to prove Corollary \ref{cor}, it is enough to prove that for all $p \geq 1$, the following
property holds:
\begin{equation*}
\forall \epsilon>0, \quad \exists p' \in \m{N}^*, \quad \forall \theta \in \Theta^p_d, \quad D(\theta,\Theta^{p'}) \leq \epsilon. 
\end{equation*}
Let us prove the property for $p=1$. Let $\epsilon \in (0,1/10]$, and assume without loss of generality that $1/\epsilon \in \m{N}^*$. Set $p'=\epsilon^{-3}$. 
Let $\theta \in \Theta^1_d=\left\{\theta(\lambda), \lambda \in (0,1] \right\}$. Let $\lambda \in (0,1]$ such that for all $m \geq 1$, $\theta_m=\lambda(1-\lambda)^{m-1}$. We distinguish between the two
following cases:
\setcounter{case}{0}
\begin{case}{$\lambda \geq \epsilon^3$}
\end{case}
 Define $\theta \in \Theta^{p'}$ by
$\theta'_m=\theta_m$ for $m \in \left\{1,...,p'-1\right\}$, $\theta'_{p'}=(1-\lambda)^{p'-1}$, and $\theta'_m=0$ for $m \geq p'+1$. We have $\left\|\theta-\theta'\right\|_1 \leq 2(1-\lambda)^{p'-1} \leq (1-\epsilon)^{\epsilon^{-3}-1} \leq \epsilon$.
\begin{case}{$\lambda < \epsilon^3$}
\end{case}
For all $r \geq 1$, define $m_r:=r\lfloor \epsilon^2/\lambda \rfloor+1$. Define $\theta \in \Theta^{p'}$ by $\theta'_m=\theta_{m_r}$ for 
all $r \in \left\{1,...,\epsilon^{-3}-1\right\}$ and $m \in \left\{m_r,...,m_{r+1}-1\right\}$, and 
$\theta_{m_{\epsilon^{-3}}}:=1-\sum_{m=1}^{m_{\epsilon^{-3}}} \theta'_m$.  
\\
By definition, for all $r \in \left\{1,...,\epsilon^{-3}-1\right\}$ and $m \in \left\{m_r,...,m_{r+1}-1\right\}$, 
we have 
\begin{equation*}
(1-\lambda)^{\lfloor \epsilon^2/\lambda \rfloor-1} \theta'_m\leq \theta_m \leq \theta'_m, 
\end{equation*}
thus
$\left|\theta_m-\theta'_m \right| \leq (1-(1-\lambda)^{\lfloor \epsilon^2/\lambda \rfloor-1}) \theta'_m \leq \epsilon^2 \theta'_m$. Moreover, we have
\begin{equation*}
\sum_{m \geq m_{\epsilon^{-3}}} \theta_m=(1-\lambda)^{m_{\epsilon^{-3}}} \leq \epsilon^2.
\end{equation*}
Hence $\left\|\theta-\theta'\right\|_1 \leq \epsilon^{2}+\epsilon^{2} \leq \epsilon$ and this proves the property for $p=1$. By a straightforward induction, the corollary holds. 
\subsection{Proof of Theorem \ref{taubbv}}
\subsubsection{Under Assumption (a)}
Let $\epsilon \in (0,1/10)$ and $\theta \in \mathcal{D}$. Thanks to Lemma \ref{comp}, we assume without loss of generality that for all $m \geq 1$, $\theta_m>0$. 
Let $t_1:=1$, and for $r \geq 1$, define recursively 
\begin{eqnarray*}
t_{r+1}:=\max \left\{ m \geq 1 \ | \ (1-\epsilon) \theta^{t_r}_m \leq \theta^{t_r}_1(1-\theta^{t_r}_1)^{m-1} \leq (1+\epsilon) \theta^{t_r}_m \right\}+t_r.
\end{eqnarray*}
with the convention $\max \emptyset=+\infty$ and $t_{r+1}=+\infty$ if $t_r=+\infty$. 
  \\
In words, between stage 1 and stage $t_r-1$, the discounted evaluation with discount factor $\theta^{t_{r-1}}_1$ is a good approximation of the evaluation $\theta^{t_{r-1}}$. 
\\
Let $\tilde{m}:=\max \left\{m \geq 1 \ | \ \sum_{m' \geq m} \theta_{m'} \geq \epsilon \right\}$ and 
$\tilde{r}:=\max \left\{r \geq 1 \ | \ t_r \leq \tilde{m} \right\}$. Define $m_r:=t_r$ for all $r \in \left\{1,...,\tilde{r}\right\}$, and $m_{\tilde{r}+1}:=\tilde{m}$.   
For $r \in \left\{1,...,\tilde{r}+1\right\}$, define $\mu^r:=\theta^{m_r}$ and
\begin{equation*}
\lambda_r:=\mu^r_1=\theta^{m_r}_1=\theta_{m_r}\left(\sum_{m \geq m_r} \theta_{m}\right)^{-1}.
\end{equation*}
Define $\pi_r \in [0,1]$ by
\begin{equation*}
\pi_r:=\prod_{r'=1}^{r-1} (1-\lambda_{r'})=\sum_{m \geq m_r} \theta_m.
\end{equation*}
\begin{lemma}
The following inequality holds: 
\begin{equation*}
\left\|v_{\theta}-v_{\lambda_1} \right\| \leq
\sum_{r=1}^{\tilde{r}} \left\|v_{\lambda_{r}}-v_{\lambda_{r+1}} \right\|+2C(2\epsilon+\theta_1).
\end{equation*}
\end{lemma}
\begin{proof}
Let $r \in \left\{1,...,\tilde{r}\right\}$. We have
\begin{equation} \label{fix21}
v_{\mu^r}=\Psi^{m_{r+1}-m_r}_{\mu^r}(v_{\mu^{r+1}}),
\end{equation}
and
\begin{equation} \label{fix22}
v_{\lambda_r}=\Psi^{m_{r+1}-m_r}_{\lambda_r}(v_{\lambda_r}).
\end{equation}
Triangle inequality yields 
\begin{eqnarray} \label{tri20}
\left\|v_{\mu^r}-v_{\lambda_r} \right\| &\leq& \left\|v_{\mu^r}-\Psi^{m_{r+1}-m_r}_{\mu^r}(v_{\lambda_{r}}) \right\|+
\left\|\Psi^{m_{r+1}-m_r}_{\mu^r}(v_{\lambda_r})-v_{\lambda_r} \right\|.
\end{eqnarray}
Note that 
\begin{eqnarray*}
\left(\sum_{m \geq m_{r+1}-m_r+1} \mu^r_m \right)
&=&  \displaystyle \left(\sum_{m \geq m_{r+1}} \theta_m\right) \left(\displaystyle \sum_{m \geq m_r} \theta_m \right)^{-1}
\\
&=& \pi_{r+1}/\pi_r.
\end{eqnarray*}
Thus, by (\ref{fix21}) and Lemma \ref{iterate} \ref{it1}, we have
\begin{eqnarray}
\nonumber
\left\|v_{\mu^r}-\Psi^{m_{r+1}-m_r}_{\mu^r}(v_{\lambda_r}) \right\| &\leq& (\pi_{r+1}/\pi_r) \left\|v_{\mu^{r+1}}-v_{\lambda_r} \right\|  
\\ \label{tri21}
&\leq& (\pi_{r+1}/\pi_r) \left\|v_{\mu^{r+1}}-v_{\lambda_{r+1}}\right\| + (\pi_{r+1}/\pi_r) \left\|v_{\lambda_{r+1}}-v_{\lambda_{r}} \right\|
\end{eqnarray}
Equations (\ref{fix22}) and Lemma \ref{iterate} \ref{it2} yield
\begin{eqnarray}
\nonumber
\left\|\Psi^{m_{r+1}-m_r}_{\mu^r}(v_{\lambda_r})-v_{\lambda_r} \right\| &\leq& C \sum_{m=1}^{m_{r+1}-m_r} \left|\mu^r_{m}-\lambda_r(1-\lambda_r)^{m-1}\right| 
\\ \nonumber
&+& \left|\sum_{m \geq m_{r+1}-m_r+1} \mu^r_{m}-\lambda_r(1-\lambda_r)^{m-1} \right| \left\|v_{\lambda_r}\right\|
\\ \label{tri22}
&\leq& 2 \epsilon C (1-\pi_{r+1}/\pi_r).
\end{eqnarray}
The last inequality stems from the definition of $(m_r)$ and the fact that $\left\|v_{\lambda_r}\right\| \leq C$. 
Combining (\ref{tri20}), (\ref{tri21}) and (\ref{tri22}) yield
\begin{equation*}
\left\|v_{\mu^r}-v_{\lambda_r} \right\| \leq 
(\pi_{r+1}/\pi_r) \left\|v_{\mu^{r+1}}-v_{\lambda_{r+1}} \right\|
+(\pi_{r+1}/\pi_r) \left\|v_{\lambda_{r+1}}-v_{\lambda_{r}} \right\|
+2\epsilon C(1-\pi_{r+1}/\pi_r),
\end{equation*}
thus
\begin{equation*}
\pi_r \left\|v_{\mu^r}-v_{\lambda_r} \right\| \leq \pi_{r+1} \left\|v_{\mu^{r+1}}-v_{\lambda_{r+1}} \right\|+ \pi_{r+1}\left\|v_{\lambda_{r+1}}-v_{\lambda_{r}}\right\|+2\epsilon C \theta_{m_r}.
\end{equation*}
Summing from $r=1$ to $r=\tilde{r}$ yields
\begin{equation*}
\left\|v_{\theta}-v_{\lambda_1} \right\| \leq \pi_{\tilde{r}+1} \left\|v_{\mu^{r_0}}-v_{\lambda_{\tilde{r}}} \right\|+\sum_{r=1}^{\tilde{r}} \pi_{r+1} \left\|v_{\lambda_{r+1}}-v_{\lambda_{r}}\right\|+2\epsilon C.
\end{equation*}
By definition of $\tilde{r}+1$, $\pi_{\tilde{r}+1} \leq \pi_{\tilde{r}+2}+\theta_{\tilde{r}+1}\leq
\epsilon+\theta_1$. We deduce that
\begin{equation*}
\left\|v_{\theta}-v_{\lambda_1} \right\| \leq \sum_{r=1}^{\tilde{r}}  \left\|v_{\lambda_{r+1}}-v_{\lambda_{r}}\right\|
+2C(2 \epsilon+\theta_1),
\end{equation*}
and the lemma is proved.
\end{proof} 
In Sorin \cite[Proposition 4, p.111]{S04}, by assumption $(\lambda_r)$ is decreasing, and because $(v_{\lambda})$ has bounded variation, the quantity $\sum_{r=1}^{\tilde{r}} \left\|v_{\lambda_{r+1}}-v_{\lambda_{r}}\right\|$ can be made as small as desired.
\\
The problem that we face is that in our framework, $(\lambda_r)$ may not be a decreasing sequence.
Nonetheless, it turns out that $(\lambda_r)$ is ``almost" decreasing, in the following sense: the number of integers $r \in \left\{1,...,\tilde{r}-1\right\}$
such that $\lambda_r<\lambda_{r+1}$ is at most of order $\epsilon^{-2}$. 
\begin{lemma}
Let $r \in \left\{1,...,\tilde{r}-1 \right\}$. We have either
\begin{itemize}
\item
$\pi_{r+1} \leq (1+\epsilon)^{-1} \pi_r$, or
\item
$\lambda_{r} \geq \lambda_{r+1}$.
\end{itemize}
\end{lemma}
\begin{proof}
By definition of $m_{r+1}$, two cases are possible:
\begin{enumerate}
\item
$(1-\epsilon) \theta^{m_r}_{m_{r+1}-m_r+1} > \theta^{m_r}_1(1-\theta^{m_r}_1)^{m_{r+1}-m_r}$ or
\item
$\theta^{m_r}_1(1-\theta^{m_r}_1)^{m_{r+1}-m_r}
>(1+\epsilon) \theta^{m_r}_{m_{r+1}-m_r+1}$
\end{enumerate}
In the first case, because $(\theta_m)_{m \geq 1}$ is decreasing, we have 
\begin{eqnarray*}
(1-\theta^{m_r}_1)^{m_{r+1}-m_r}
\leq (1-\epsilon),
\end{eqnarray*}
Moreover, by definition of $(m_r)$, for all $m \in \left\{1,...,m_{r+1}-m_r\right\}$,
\begin{equation*}
(1+\epsilon) \theta^{m_r}_m \geq \theta^{m_r}_1(1-\theta^{m_r}_1)^{m-1}.
\end{equation*}
Combining these two facts, we get
\begin{eqnarray*}
(1+\epsilon)\sum_{m=1}^{m_{r+1}-m_r} \theta^{m_r}_m
&\geq& \sum_{m=1}^{m_{r+1}-m_r} 
\theta^{m_r}_1(1-\theta^{m_r}_1)^{m-1}
\\
&=&1-(1-\theta^{m_r}_1)^{m_{r+1}-m_r}
\\
& \geq & \epsilon.
\end{eqnarray*}
Consequently,
\begin{eqnarray*}
(1+\epsilon)(1-\pi_{r+1}/\pi_r) \geq \epsilon,
\end{eqnarray*}
thus
\begin{equation*}
\pi_{r+1} \leq \frac{1}{1+\epsilon} \pi_r.
\end{equation*}
In the second case, we have 
\begin{equation*}
\theta^{m_r}_1 \geq (1+\epsilon) \theta^{m_r}_{m_{r+1}-m_r+1},
\end{equation*}
and we deduce that
\begin{eqnarray*}
\pi_r \lambda_{r} \geq (1+\epsilon) \pi_{r+1} \lambda_{r+1}.
\end{eqnarray*}
We deduce that either $\pi_r \geq (1+\epsilon) \pi_{r+1}$, or $\lambda_r \geq \lambda_{r+1}$. 
\end{proof}
Note that $(1+\epsilon)^{-\epsilon^{-2}-1}<\epsilon$ and $\pi_{\tilde{r}} \geq \epsilon$. Consequently, there exists $k \in \m{N}^*$ satisfying $k \leq \epsilon^{-2}$, and $r_1,...,r_k \in \left\{1,...,\tilde{r} \right\}$ such that $r_1=1$, $r_k=\tilde{r}+1$ and
for all $k' \in \left\{1,...,k-1\right\}$, for all $r \in \left\{r_{k'},...,r_{k'+1}-1\right\}$, 
$\lambda_{r+1} \leq \lambda_{r}$. To conclude, let $\beta \in (0,1]$ such that
for all decreasing sequences $(d_r)_{r \geq 1} \in (0,\beta]^{\m{N}^*}$,  
\begin{equation*}
\sum_{r=1}^{+\infty} \left\|v_{d_{r+1}}-v_{d_r} \right\| \leq \epsilon^{3}.
\end{equation*}
Set $\alpha:=\beta \epsilon$, and assume that $\theta_1 \leq \alpha$. Let $r \in \left\{1,...,\tilde{r} \right\}$. Then $\lambda_r \leq \beta \epsilon/\epsilon=\beta$, and thus
\begin{equation*}
\sum_{r=1}^{\tilde{r}}  \left\|v_{\lambda_{r+1}}-v_{\lambda_{r}}\right\| \leq k \epsilon^3 \leq \epsilon,
\end{equation*}
and the theorem is proved.
\subsubsection{Under Assumption (b)}
The proof is a simple adaptation of the proof in \cite[Theorem 3]{Z14}.
Let $\theta \in \mathcal{D}$. For $r \geq 1$, define $\lambda_r:=\theta^r_1=\theta_{r}/\sum_{m' \geq r} \theta_{m'}$. 
We have
\begin{equation*} 
v_{\theta^r}=\Psi^1_{\theta^r}(v_{\theta^{r+1}}),
\end{equation*}
and
\begin{equation*} 
v_{\lambda_r}=\Psi^1_{\lambda_r}(v_{\lambda_r})=\Psi^1_{\theta_r}(v_{\lambda_r}).
\end{equation*}
These two equations and Lemma \ref{iterate} \ref{it1} yield
\begin{eqnarray*}
\left\|v_{\theta^r}-v_{\lambda_r} \right\| &\leq& (1-\lambda_r) \left\|v_{\theta^{r+1}}-v_{\lambda_r} \right\|.
\end{eqnarray*}
This equation is the same as Equation (11) in \cite[p.9]{Z14}. After that, the proof is identical.

\section{Applications to zero-sum stochastic games} \label{app}

\subsection{Zero-sum stochastic games}
Consider a general model of zero-sum stochastic game, as described in Maitra and Parthasarathy \cite{MP70}.
It is described by a state space $K$, which is a Borel subset of a Polish space, a measurable action set $I$ (resp. $J$) for Player 1 (resp. 2), that is a Borel subset of a Polish space, a Borel measurable transition $q: K \times I \times J \rightarrow \Delta(K)$, and a bounded Borel measurable payoff function $g:K \times I \times J \rightarrow \m{R}$.

Given an initial state $k_1 \in K$, the stochastic game $\Gamma^{k_1}$ that starts in $k_1$ proceeds as follows. At each stage $m \geq 1$, both players choose simultaneously and independently an action, $i_m \in I$ (resp. $j_m \in J$) for Player 1 (resp. 2). The payoff at stage $m$ is $g_m:=g(k_m,i_m,j_m)$.
The state $k_{m+1}$ of stage $m+1$ is drawn from the probability distribution $q(k_m,i_m,j_m)$. Then $(k_{m+1},i_m,j_m)$ is made public to players.
\\
The set of all possible histories before stage $m$ is
$H_m:=(K \times I \times J)^{m-1} \times K$. A \textit{behavioral strategy} for Player 1 (resp. 2) is a Borel measurable mapping $\displaystyle \sigma:\cup_{m \geq 1} H_m \rightarrow \Delta(I)$ (resp. $\displaystyle \tau:\cup_{m \geq 1} H_m \rightarrow \Delta(J)$).
\\
A triple $(k_1,\sigma,\tau) \in K \times \Sigma \times \mathcal{T}$ naturally induces a probability measure on $H_\infty:=(K \times I  \times J)^{\m{N}^*}$, denoted by $\mathbb{P}^{k_1}_{\sigma,\tau}$. Let $\theta \in \Delta(\m{N}^*)$. The $\theta-weighted \ game$ $\Gamma^{k_1}_\theta$ is the game defined by its normal form $(\Sigma,\mathcal{T},\gamma_{\theta}^{k_1})$, where 
\begin{equation*}
\gamma^{k_1}_{\theta}(\sigma,\tau):=\mathbb{E}^{k_1}_{\sigma,\tau}\left(\sum_{m \geq 1} \theta_m g_m \right).
\end{equation*}
When $\theta=\theta(n)$ for some $n \in \m{N}^*$, the game $\Gamma_n:=\Gamma_{\theta}$ is called the \textit{$n$-stage game}, and its payoff function is denoted by $\gamma_n:=\gamma_{\theta}$. When $\displaystyle \theta=\theta(\lambda)$ for some $\lambda \in (0,1]$, the game $\Gamma_{\lambda}:=\Gamma_{\theta}$ is called the \textit{$\lambda$-discounted game}, and its payoff function is denoted by $\gamma_{\lambda}$.
Let $f: K \rightarrow \m{R}$ be a bounded Borel measurable function, and $(k,x,y) \in K \times \Delta(I) \times \Delta(J)$. Define 
\begin{equation*} \label{deftrans}
\m{E}^{k}_{x,y}(f):=\int_{(k',i,j) \in K \times I \times J} f(k') dq(k,i,j)(k')dx(i) dy(j)
\end{equation*}
and 
\begin{equation*} \label{defpayoff}
\displaystyle g(k,x,y):=\int_{(i,j) \in I \times J} g(k,i,j) dx(i) dy(j) .
\end{equation*}
We make the following assumption:
\begin{assumption} \label{assdyn}
For all $k_1 \in K$ and $\theta \in \Delta(\m{N}^*)$, the game $\Gamma^{k_1}_{\theta}$ has a value, that is, there exists a real number $v_{\theta}(k_1)$ such that:
\begin{equation*}
v_{\theta}(k_1)=\sup_{\sigma \in \Sigma} \inf_{\tau \in \mathcal{T}} \gamma^{k_1}_{\theta}(\sigma,\tau)
=\inf_{\tau \in \mathcal{T}} \sup_{\sigma \in \Sigma} \gamma^{k_1}_{\theta}(\sigma,\tau).
\end{equation*}
Moreover, for all $k_1 \in K$ and $\theta \in \Delta(\m{N}^*)$, the mapping $(v_{\theta})$ is Borel measurable and satisfies:
\begin{eqnarray} \label{dyn11}
v_{\theta}(k_1)&=& \sup_{x \in \Delta(I)} \inf_{y \in \Delta(J)}
\left\{ \theta_1 g(k_1,x,y)+(1-\theta_1)\m{E}^{k_1}_{x,y}(v_{\hat{\theta}}) \right\}
\\
\label{dyn12}
&=& \inf_{y \in \Delta(J)} \sup_{x \in \Delta(I)} 
\left\{ \theta_1 g(k_1,x,y)+(1-\theta_1)\m{E}^{k_1}_{x,y}(v_{\hat{\theta}}) \right\}.
\end{eqnarray}
\end{assumption}
The game $\Gamma$ has an \textit{asymptotic value} $v^*$ if both $(v_n)$ and $(v_{\lambda})$ converge uniformly to $v^*$.
Let $X$ be the set of bounded Borel measurable functions from $K$ to $\m{R}$, equipped with the uniform norm, and for all $(f,k) \in X \times K$, define
\begin{equation*}
 \Psi(\lambda,f)(k):=\sup_{x \in \Delta(I)} \inf_{y \in \Delta(J)} \left\{\lambda g(k,x,y)+(1-\lambda)\m{E}^k_{x,y}(f) \right\}. 
\end{equation*}
 We make the following assumption:
\begin{assumption} \label{assmes}
For all $\lambda \in [0,1]$ and $f \in X$, $\Psi(\lambda,f)$ is Borel measurable.
\end{assumption}
\begin{remark}
When $I$ and $J$ are compact metric spaces, and for all $k \in K$, $q(k,.)$ and $g(k,.)$ are jointly continuous, then Assumptions 2 and 3 hold (see \cite[Proposition VII.1.4, p.394]{MSZ}). Maitra and Parthasarathy \cite{MP70} provide even weaker conditions under which Assumptions 2 and 3 hold in the $n$ stage-case and the $\lambda$-discounted case, which generalize immediately to the $\theta$-weighted case.
\end{remark}

\subsection{Tauberian theorem for value functions}
Assume that Assumptions \ref{assdyn} and \ref{assmes} are in force. 
\\
The set $X$ is a Banach space, and Assumption 3 ensures that $\Psi$ is well defined from $[0,1] \times X$ to $X$. Moreover, $\Psi$ satisfies (\ref{nonexp}) and Assumption \ref{assop} for $C:=\sup_{(k,i,j) \in K \times I \times J} \left|g(k,i,j)\right|$. Thus, Theorems \ref{taubgen} and \ref{taubbv} apply to $\Psi$. By Assumption 2, the family of values $(v_{\theta})$ satisfies equation (\ref{recop}). Consequently, the two following theorems hold: 
\begin{theorem} 
Assume that $\Gamma$ has an asymptotic value $v^*$.
Then for all $\epsilon>0$, for all $p \in \m{N}^*$, there exists $\alpha>0$ such that for all $\theta \in \Delta(\m{N}^*)$,
\begin{equation*}
 I_p(\theta) \leq \alpha \Rightarrow
\left\|v_{\theta}-v^* \right\|  \leq \epsilon.
\end{equation*}
\end{theorem}
\begin{theorem} 
Assume that $(v_{\lambda})$ has bounded variation. Then $\Gamma$ has an asymptotic value $v^*$. Moreover, for all $\epsilon>0$, there exists $\alpha>0$ such that for all $\theta \in \mathcal{D}$,
\begin{equation*}
 \sup_{m \geq 1} \theta_m \leq \alpha \Rightarrow
\left\|v_{\theta}-v^* \right\|  \leq \epsilon.
\end{equation*}
\end{theorem}
The above theorem was already known. Indeed, if $(v_{\lambda})$ has bounded variation, then the uniform value exists (see \cite{MN81}). This implies the above theorem (see Theorem 1 and Remark 4 in Neyman and Sorin \cite{NS10}). Nonetheless, the alternative proof that is provided in this paper is shorter and more elementary. 
\subsection{Asymptotically optimal strategies in finite stochastic games} \label{secfrac}
Assume that $K$, $I$ and $J$ are finite. Because $(v_{\lambda})$ can be expanded in Puiseux series (see \cite{BK76}), there exists $C>0$ and $s>0$ such that for all $\lambda,\lambda' \in (0,1]$, 
\begin{equation} \label{puiseux}
\left\|v_{\lambda}-v_{\lambda'}\right\| \leq C\left|\lambda^s-{\lambda'}^s \right|.
\end{equation}
Let $k \in K$. Denote by $X^*_{\lambda}(k)$ the set of $x \in \Delta(I)$ such that
\begin{eqnarray*} 
v_{\lambda}(k)&=&  \inf_{y \in \Delta(J)}
\left\{ \lambda g(k,x,y)+(1-\lambda)\m{E}^k_{x,y}(v_{\lambda}) \right\}.
\end{eqnarray*}
Denote by $Y^*_{\lambda}(k)$ the set of $y \in \Delta(J)$ such that
\begin{eqnarray*} 
v_{\lambda}(k)&=&  \sup_{x \in \Delta(I)}
\left\{ \lambda g(k,x,y)+(1-\lambda)\m{E}^k_{x,y}(v_{\lambda}) \right\}.
\end{eqnarray*}

\begin{definition}
Let $\theta \in \Delta(\m{N}^*)$. As in the previous section, for $m \in \m{N}^*$, define $\lambda_m:=\theta_m/\sum_{m' \geq m} \theta_{m'}$ if $\theta_m \neq 0$, and $\lambda_m:=0$ otherwise. A strategy $\sigma^{\theta} \in \Sigma$ for Player 1 is $\theta$-\textit{discounted} if for all $ m \geq 1$ and $h_m \in H_m$, $\sigma(h_m) \in X^*_{\lambda_m}(k_m)$. 
Similarly, a strategy $\tau \in \mathcal{T}$ for Player 2 is $\theta$-\textit{discounted} if for all $ m \geq 1$ and $h_m \in H_m$, $\tau(h_m) \in Y^*_{\lambda_m}(k_m)$. 
\end{definition}
The following theorem shows that $\theta$-discounted strategies are asymptotically optimal in $\Gamma_{\theta}$, as $\theta_1$ goes to 0. 
\begin{theorem}
For all $\epsilon>0$, there exists $\alpha>0$ such that for all $\theta \in \mathcal{D}$, for all
$k_1 \in K$, for all $\theta$-discounted strategies $\sigma^{\theta} \in \Sigma$ and $\tau^{\theta} \in \mathcal{T}$, for all $\sigma \in \Sigma$ and $\tau \in \mathcal{T}$, 
\begin{equation*}
 \sup_{m \geq 1} \theta_m=\theta_1 \leq \alpha \Rightarrow
\gamma^{k_1}_{\theta}(\sigma^{\theta},\tau)\geq v^*(k_1)-\epsilon \quad \text{and} \quad
\gamma^{k_1}_{\theta}(\sigma,\tau^{\theta})  \leq v^*(k_1)+\epsilon.
\end{equation*}

\end{theorem}
\begin{proof}
Let us prove the result for Player 1, the result for Player 2 follows by symmetry. 
Let $X$ be the set of  functions from $K$ to $\m{R}$, equipped with the uniform norm, and for all $(f,k) \in X \times K$, define
\begin{equation*}
 \Psi(\lambda,f)(k):=\inf_{x \in X_{\lambda}^*(k)} \inf_{y \in \Delta(J)} \left\{\lambda g(k,x,y)+(1-\lambda)\m{E}^k_{x,y}(f) \right\}.
 \end{equation*}
 The operator $\Psi$ satisfies equations (\ref{nonexp}) and (\ref{reg}). 
 \\
For $k \in K$ and $\theta \in \mathcal{D}$, define
\begin{equation*}
w_{\theta}(k):=\inf_{\sigma \ \theta-\text{discounted}} \inf_{\tau \in \mathcal{T}}\gamma^{k}_{\theta}(\sigma,\tau). 
\end{equation*}
For all $\theta \in \Delta(\m{N}^*)$, we have
\begin{equation*}
w_{\theta}=\Psi(\theta_1,w_{\theta}).
\end{equation*}
Moreover, for all $\lambda \in (0,1]$, $w_{\lambda}=v_{\lambda}$, where $v_{\lambda}$ is the value of the $\lambda$-discounted game.  
Applying Theorem \ref{taubbv} to $\Psi$ proves the result.
\end{proof}
\begin{remark}
To the best of our knowledge, the above result was not known even when $\theta$ is restricted to be a $n$-stage evaluation. 
\end{remark}
\begin{remark}
For each $\lambda \in (0,1]$, there exists a stationary optimal strategy $s^*(\lambda)$ (a strategy that depends only on the current state) that is optimal in the game $\Gamma_{\lambda}(k_1)$, for any $k_1 \in K$ (see \cite{SH53}). For $\theta \in \mathcal{D}$, consider the following strategy in $\Gamma_{\theta}$: at stage $m$, play the strategy $s^*(\lambda_m)$. By the previous theorem, this strategy is asymptotically optimal, as $\theta_1$ goes to 0. Thus, this theorem provides a simple way to build asymptotically optimal strategies in weighted games.
\end{remark}
\section{An example} \label{secex}
In this section, we provide an example of a 1-Player stochastic game (Markov Decision Process) that has an asymptotic value, and such that there exists $(\theta^N) \in \mathcal{D}^{\m{N}^*}$ satisfying $\lim_{N \rightarrow +\infty} \theta^N_1=0$, but such that $(v_{\theta^N})$ does not converge. For all $N \geq 1$, $\theta^N$ is piecewise constant, but the number of pieces grows to infinity as $N$ goes to infinity. This shows first that even if we restrict to decreasing sequences of weights, it is crucial that in Theorem \ref{taubgen}, the number of pieces is bounded. Second, this example shows that in Theorem \ref{taubbv}, the bounded variation assumption can not be relaxed. Third, this gives a negative answer to a question raised in Renault \cite{R14}. 
\\

Let $\Gamma$ be the following MDP: the set space is $K:=\left\{0,1 \right\} \times \m{N}^* \cup \left\{0^*\right\}$, and the action set is $I:=\left\{C,Q\right\}$. The payoff function is equal to $0$ in states $\left\{0\right\} \times \m{N}^*$, and equal to $1$ in states $\left\{1 \right\} \times \m{N}^*$. The state $0^*$ is an absorbing state: once in $0^*$, the game remains forever in it, and the payoff is 0. 
\\
Let $M_0:=10^{(10^{10})}$. The Dirac measure at $k \in K$ is denoted by $\delta_k$. 
The transition function $q$ is the following: for all $m \in \m{N}^*$,
\begin{eqnarray*}
&&q((0,m),C):=(0,m+1)
\\
&&q((0,m),Q):=[1-(1/\ln(\ln(\ln(m+M_0)))] \cdot \delta_{(1,m^2)}+1/\ln(\ln(\ln(m+M_0))) \cdot \delta_{0^*}
\\
&&q((1,m),.):=(1,m-1)
\\
&&q((1,1),.):=(0,1).
\end{eqnarray*}
The idea of the game starting from state $(0,1)$ is the following: the decision-maker plays $C$ for a certain period of time $m$, and meanwhile gets payoff 0, then he plays $Q$. At this point, the state is absorbed in $0^*$ with probability $1/\ln(\ln(\ln(m+m_0)))$, otherwise goes to state $(1,m^2)$. Then, he enjoys a payoff 1 during $m^2$ stages, and afterwards the game is again in state $(0,1)$. 
\\
Let us first prove that $\Gamma$ has an asymptotic value:
\begin{proposition}
The family $(v_n)$ converges uniformly to the mapping $v^*$ such that $v^*(k)=1$ for all $k \neq 0^*$ and $v^*(0^*)=0$. Consequently, $\Gamma$
has an asymptotic value.
\end{proposition}
\begin{proof}
For $n \in \m{N}^*$, consider the following stationary strategy in the $n$-stage game: play $C$ until the state is $(0,m)$ with $m \geq n^{1/2}$, then play $Q$, and afterwards play $C$ forever. Under such a strategy, whatever be the initial state, the expected payoff is larger than 
\begin{equation*}
(1-n^{-1/2}-n^{-1})\left[1-1/\ln(\ln(\ln(n^{1/2}+M_0)))\right].
\end{equation*}
This quantity goes to 1 as $n$ goes to infinity. 
\end{proof}
\begin{proposition}
There exists a sequence $(\theta^N) \in \mathcal{D}^{\m{N}^*}$ such that $(\theta^N_1)$ goes to 0 but the real sequence $(v_{\theta^N}(0,1))$ converges to 0. 
\end{proposition}
\begin{remark}
This result answers negatively the open question mentioned in Renault \cite[Section 5]{R14}: indeed, the total variation of $\theta^N$ vanishes as $N$ goes to infinity, but $(v_{\theta^N})$ does not converge to the asymptotic value. 
\end{remark}
\begin{proof}
For $N \in \m{N}^*$ and $r \in \left\{1,...,N+1 \right\}$, define 
$\displaystyle m^N_r:=\sum_{r'=1}^{r} N^{3^{r'}}-N^{3}+1$.
Consider $\theta^N \in \Delta(\m{N}^*)$ defined by 
\begin{equation*}
\theta^N_m:= \left\{
\begin{array}{ll}
\displaystyle N^{-3^{r+1}-1} & \mbox{if} \ \ r \in \left\{1,...,N\right\} \ \text{and} \ m^N_r \leq m <m^N_{r+1}, \\
0 & \mbox{if} \ \ m \geq m^N_{N+1}.
\end{array}
\right.
\end{equation*}
Consider $\Gamma_{\theta^{N}}(0,1)$, and fix a strategy for the decision-maker. Let $r \in \left\{1,2,...,N\right\}$. We distinguish between the two following cases:
\setcounter{case}{0}
\begin{case}{The decision-maker does not play $Q$ in a state $(0,m) \in \left\{0 \right\} \times \m{N}^*$ between stages $m_r$ and $m_{r+1}-1$}
\end{case}
In this case, the only way that his total payoff between stages $m_r$ and $m_{r+1}-1$ is positive is that $k_{m_r}=(1,\tilde{m})$, for some $\tilde{m} \in \m{N}^*$. By definition of the transition, we have
\begin{equation*}
\tilde{m} \leq \left(\sum_{r'=1}^r N^{3^{r'}}\right)^2 \leq N^{3^{r+1}-1}.
\end{equation*}
We deduce that the total payoff between stages $m_r$ and $m_{r+1}-1$ is smaller than 
\\
$\tilde{m} N^{-3^{r+1}-1} \leq (1/N)^2$. 
\begin{case}{The decision-maker plays $Q$ in a state $(0,m) \in \left\{0 \right\} \times \m{N}^*$ between stages $m_r$ and $m_{r+1}-1$}
\end{case}
By definition of the transition, we have $m \leq \sum_{r'=1}^{r} N^{3^{r'}} \leq N^{3^{r+1}} \leq N^{3^{N+1}}$, thus 
\\
$1/\ln(\ln(\ln(m+M_0))) \geq (10\ln(N))^{-1} \geq N^{-1/4}$. 
\\
\\
Let us now finish the proof. Assume $N \geq 10$. For $m \in \m{N}^*$, let $p(m)$ be the probability that the state is $0^*$ at stage $m$. Assume that under the strategy of the decision-maker, $p(m^N_{N+1}) \geq N^{-1}$. 
Because $(1-N^{-1/4})^{\sqrt{N}} <N^{-1}$, this implies that the number of $r$ that satisfy Case 2 is smaller than $\sqrt{N}$. Thus, the total payoff in $\Gamma_{\theta}(0,1)$ is smaller than $N^{-1}+N^{-1/2}$.
\\
\\
Assume now that under the strategy of the decision-maker, $p(m^N_{N+1}) < N^{-1}$.
Let $m$ be the first stage such that $p(m) \leq N^{-1}$. Because $(1-N^{-1/4})^{\sqrt{N}}<N^{-1}$, the number of $r$ that satisfy Case 2 and such that $m_r \leq m$ is smaller than $\sqrt{N}$. Like before, we deduce that the total payoff between stages 1 and $m$ is smaller than $N^{-1}+N^{-1/2}$, and afterwards it is smaller than $N^{-1}$ by definition of $m$. This proves that $\lim_{N \rightarrow+\infty} v_{\theta^N}(0,1)=0$.   
\end{proof}
\section*{Acknowledgments}
The author is very grateful to Miquel Oliu-Barton, J\'er\^ome Renault, Sylvain Sorin and Guillaume Vigeral for helpful discussions.
\bibliography{bibliogen}

\begin{thebibliography}{10}

\bibitem{BK76}
T.~Bewley and E.~Kohlberg.
\newblock The asymptotic theory of stochastic games.
\newblock {\em Mathematics of Operations Research}, 1(3):197--208, 1976.

\bibitem{JN16}
A.~Jaskiewicz and A.S. Nowak.
\newblock Zero-sum stochastic games.
\newblock {\em preprint}, 2016.

\bibitem{K14}
D.~Khlopin.
\newblock On uniform tauberian theorems for dynamic games.
\newblock {\em preprint arXiv:1412.7331}, 2014.

\bibitem{K15}
D.~Khlopin.
\newblock On asymptotic value for dynamic games with saddle point.
\newblock {\em preprint arXiv:1501.06933}, 2015.

\bibitem{LS15}
Rida Laraki and Sylvain Sorin.
\newblock Advances in zero-sum dynamic games.
\newblock {\em Handbook of Game Theory}, 4:27--95, 2015.

\bibitem{LS92}
E.~Lehrer and S.~Sorin.
\newblock A uniform tauberian theorem in dynamic programming.
\newblock {\em Mathematics of Operations Research}, 17(2):303--307, 1992.

\bibitem{LV16}
X.~Li and X.~Venel.
\newblock Recursive games: uniform value, tauberian theorem and the mertens
  conjecture "maxmin=lim v(n)= lim v(lambda)".
\newblock {\em Published online in International Journal of Game Theory}, 2015.

\bibitem{MP70}
A~Maitra and T~Parthasarathy.
\newblock On stochastic games.
\newblock {\em Journal of Optimization Theory and Applications}, 5(4):289--300,
  1970.

\bibitem{MS93}
A~Maitra and W~Sudderth.
\newblock Borel stochastic games with lim sup payoff.
\newblock {\em The Annals of Probability}, pages 861--885, 1993.

\bibitem{MN81}
J.F. Mertens and A.~Neyman.
\newblock Stochastic games.
\newblock {\em International Journal of Game Theory}, 10(2):53--66, 1981.

\bibitem{MSZ}
J.F. Mertens, S.~Sorin, and S.~Zamir.
\newblock {\em Repeated games}.
\newblock CORE DP 9420-22, 1994.

\bibitem{NS03b}
A.~Neyman and S.~Sorin.
\newblock {\em Stochastic games and applications}.
\newblock Kluwer Academic Publishers, 2003.

\bibitem{NS10}
A.~Neyman and S.~Sorin.
\newblock Repeated games with public uncertain duration process.
\newblock {\em International Journal of Game Theory}, 39(1):29--52, 2010.

\bibitem{BV13}
M.~Oliu-Barton and G.~Vigeral.
\newblock A uniform tauberian theorem in optimal control.
\newblock In {\em Advances in Dynamic Games}, pages 199--215. Birkha\"{u}ser,
  2013.

\bibitem{R14}
J.~Renault.
\newblock General limit value in dynamic programming.
\newblock {\em Journal of Dynamics and Games}, 1(3):471--484, 2014.

\bibitem{RS01}
D.~Rosenberg and S.~Sorin.
\newblock An operator approach to zero-sum repeated games.
\newblock {\em Israel Journal of Mathematics}, 121(1):221--246, 2001.

\bibitem{SH53}
L.S. Shapley.
\newblock Stochastic games.
\newblock {\em Proceedings of the National Academy of Sciences of the United
  States of America}, 39(10):1095--1100, 1953.

\bibitem{SZ16}
E.~Solan and B.~Ziliotto.
\newblock Stochastic games with signals.
\newblock In {\em Advances in Dynamic and Evolutionary Games}, pages 77--94.
  Springer, 2016.

\bibitem{sorin02b}
S.~Sorin.
\newblock {\em A first course on zero-sum repeated games}, volume~37.
\newblock Math\'{e}matiques et Applications, \ Springer, 2002.

\bibitem{S04}
S.~Sorin.
\newblock Asymptotic properties of monotonic nonexpansive mappings.
\newblock {\em Discrete Event Dynamic Systems}, 14(1):109--122, 2004.

\bibitem{vigeral13}
G.~Vigeral.
\newblock A zero-sum stochastic game with compact action sets and no asymptotic
  value.
\newblock {\em Dynamic Games and Applications}, 3(2):172--186, 2013.

\bibitem{Z15}
B.~Ziliotto.
\newblock A tauberian theorem for nonexpansive operators and applications to
  zero-sum stochastic games.
\newblock {\em preprint arXiv:1501.06525, to appear in Mathematics of
  Operations Research}, 2015.

\bibitem{Z14}
B.~Ziliotto.
\newblock General limit value in zero-sum stochastic games.
\newblock {\em International Journal of Game Theory}, 45(1):353--374, 2016.

\bibitem{Z13}
B.~Ziliotto.
\newblock Zero-sum repeated games: counterexamples to the existence of the
  asymptotic value and the conjecture maxmin= lim v (n).
\newblock {\em The Annals of Probability}, 44(2):1107--1133, 2016.

\end{thebibliography}
\end{document}